\documentclass[11pt]{amsart}
\usepackage{tikz,amsthm,amsmath,amstext,amssymb,amscd,epsfig,euscript, mathrsfs, dsfont,pspicture,multicol,graphpap,graphics,graphicx,times,enumerate,subfig,sidecap,wrapfig,color}
\usepackage{amsmath}
\usepackage{amssymb}
\usepackage{pgf}

\usepackage[colorlinks,citecolor=red]{hyperref}

\newcommand{\R}{{\mathbb R}}       % Field of real numbers
\newcommand{\CC}{{\mathcal C}}

\newcommand{\HH}{{\mathcal H}}

\newcommand{\RR}{{\mathcal R}}

\newcommand{\EE}{{\mathcal E}}

\newcommand{\diam}{\mathop{\rm diam}}
\newcommand{\dist}{{\rm dist}}

\newcommand{\rf}[1]{{(\ref{#1})}}

\newcommand{\supp}{\operatorname{supp}}

\newcommand{\vphi}{{\varphi}}
\newcommand{\ve}{{\varepsilon}}
\newcommand{\vv}{{\vspace{2mm}}}
\newcommand{\vvv}{\vspace{4mm}}
\newcommand{\wt}[1]{{\widetilde{#1}}}

\def\Xint#1{\mathchoice
{\XXint\displaystyle\textstyle{#1}}%
{\XXint\textstyle\scriptstyle{#1}}%
{\XXint\scriptstyle\scriptscriptstyle{#1}}%
{\XXint\scriptscriptstyle\scriptscriptstyle{#1}}%
\!\int}
\def\XXint#1#2#3{{\setbox0=\hbox{$#1{#2#3}{\int}$ }
\vcenter{\hbox{$#2#3$ }}\kern-.58\wd0}}

\def\avint{\Xint-}

\newtheorem{theorem}{Theorem}[section]
\newtheorem{lemma}[theorem]{Lemma}
\newtheorem{remark}[theorem]{Remark}

\newtheorem*{lemma*}{Lemma}
\newtheorem*{theorem*}{Theorem}

\theoremstyle{definition}
\newtheorem{definition}[theorem]{Definition}

\theoremstyle{remark}
\newtheorem{rem}[theorem]{\bf Remark}

\numberwithin{equation}{section}

\newcommand{\cnj}[1]{\overline{#1}}
\newcommand{\RRem}{\begin{rem}}
\newcommand{\erem}{\end{rem}}

\def\d{\partial}

\makeatletter
\def\@tocline#1#2#3#4#5#6#7{\relax
  \ifnum #1>\c@tocdepth % then omit
  \else
    \par \addpenalty\@secpenalty\addvspace{#2}%
    \begingroup \hyphenpenalty\@M
    \@ifempty{#4}{%
      \@tempdima\csname r@tocindent\number#1\endcsname\relax
    }{%
      \@tempdima#4\relax
    }%
    \parindent\z@ \leftskip#3\relax \advance\leftskip\@tempdima\relax
    \rightskip\@pnumwidth plus4em \parfillskip-\@pnumwidth
    #5\leavevmode\hskip-\@tempdima
      \ifcase #1
       \or\or \hskip 1em \or \hskip 2em \else \hskip 3em \fi%
      #6\nobreak\relax
    \dotfill\hbox to\@pnumwidth{\@tocpagenum{#7}}\par
    \nobreak
    \endgroup
  \fi}
\makeatother

\def\cF{{\mathscr{F}}}

\def\cH{{\mathcal{H}}}

\def\cM{{\mathscr{M}}}

\def\bB{{\mathbb{B}}}

\def\bR{{\mathbb{R}}}

\def\bN{{\mathbb{N}}}

\newcommand{\ps}[1]{\left( #1 \right)}

\newcommand{\ck}[1]{\left\{#1 \right\}}

\def\gec{\gtrsim}
\def\lec{\lesssim}

\def\grad{\nabla}

\def\Tan{\textrm{Tan}}
\def\bR{\mathbb{R}}
\def\lec{\lesssim}

\def\ext{\textrm{ext}}

\newcommand\eqn[1]{\eqref{e:#1}}

\newcommand\Theorem[1]{Theorem \ref{t:#1}}
\newcommand\Lemma[1]{Lemma \ref{l:#1}}

\begin{document}

\title[A two-phase problem for harmonic measure]{On a two-phase problem for harmonic measure in general domains}

\author[Azzam]{Jonas Azzam}
\address{School of Mathematics \\ University of Edinburgh \\ JCMB, Kings Buildings \\
Mayfield Road, Edinburgh,
EH9 3JZ, Scotland.}
\email{j.azzam "at" ed.ac.uk}

\newcommand{\jonas}[1]{\marginpar{\color{magenta} \scriptsize \textbf{Jonas:} #1}}

\author[Mourgoglou]{Mihalis Mourgoglou}

\address{Mihalis Mourgoglou
\\
Departament de Matem\`atiques\\
 Universitat Aut\`onoma de Barcelona
\\
%Universitat Autònoma de Barcelona
Edifici C Facultat de Ci\`encies
\\
08193 Bellaterra (Barcelona), Catalonia
}
\email{mourgoglou@mat.uab.cat}

\author[Tolsa]{Xavier Tolsa}
\address{Xavier Tolsa
\\
ICREA, Passeig Lluís Companys 23 08010 Barcelona, Catalonia, and\\
Departament de Matem\`atiques and BGSMath
\\
Universitat Aut\`onoma de Barcelona
\\
Edifici C Facultat de Ci\`encies
\\
08193 Bellaterra (Barcelona), Catalonia
}
\email{xtolsa@mat.uab.cat}

\author[Volberg]{Alexander Volberg}
\address{Alexander Volberg
\\
Department of Mathematics
\\
Michigan State University
\\
East Lan\-sing, MI 48824, USA}

\email{volberg@math.msu.edu}

\thanks{The first three authors were supported by the ERC grant 320501 of the European Research Council (FP7/2007-2013). X.T. was also supported by 2014-SGR-75 (Catalonia), MTM2013-44304-P (Spain), and by the Marie Curie ITN MAnET (FP7-607647). A.V. was supported by NSF Grant DMS1600065}

\begin{abstract}
We show that, for disjoint domains in the Euclidean space, mutual absolute continuity of their harmonic measures implies  absolute continuity with respect to surface measure and rectifiability in the intersection of their boundaries. This improves on our previous result which assumed that the boundaries satisfied the capacity density condition. 
\end{abstract}

\maketitle

\tableofcontents

\section{Introduction}

Let $\Omega_1,\Omega_2\subset\R^{n+1}$ be disjoint domains, and let $E\subset \partial \Omega_1\cap 
\partial \Omega_2$. In 1990 C. Bishop \cite{Bishop-questions} conjectured that if the respective harmonic measures of $\Omega_1$ and
$\Omega_2$ are mutually absolutely continuous, then they should be also mutually absolutely continuous with respect the $n$-dimensional Hausdorff measure on an $n$-rectifiable set  $\R^{n+1}$.
In the work \cite{AMT} we proved this conjecture under the assumption that $\Omega_1$ and $\Omega_2$
satisfy the so called capacity density condition (CDC) (although we obtained the stronger property that the set of tangent points for $\d\Omega_{1}$ has positive $\cH^{n}$-measure). In the present paper we prove this in full
generality. The precise result reads as follows.

\begin{theorem}\label{t:main}
For $n\geq 2$, let $\Omega_1,\Omega_2\subset \R^{n+1}$ be two domains and denote by 
$\omega_1$ and $\omega_2$ their respective harmonic measures. Let $E\subset \partial \Omega_1\cap
\partial \Omega_2$ 
 be a Borel set such that $\omega_{1}|_E\ll\omega_{2}|_E\ll\omega_{1}|_E$. Then $E$ contains an $n$-rectifiable subset $F$ with $\omega_1(E\setminus F)=0$ such that
$\omega_1|_F$ and $\omega_2|_F$ are mutually absolutely continuous with respect to $\cH^{n}|_F$. 
\end{theorem}

We remark that, in the planar case $n=1$, the analogous conclusion had been proved previously by Bishop in \cite{Bishop-arkiv}. Another partial result was obtained by Kenig, Preiss and Toro in
\cite{KPT}. Therein, they showed, among others, that if $\Omega_1$ and $\Omega_2={\ext}(\Omega_1)$
are NTA domains with mutually absolutely continuous harmonic measures, then these
harmonic measures are concentrated on a set of dimension $n$.

As in \cite{AMT}, the main tools to prove the preceding result stated in Theorem \ref{t:main} are:
\begin{itemize}
\item a blowup argument for harmonic measure
inspired by the techniques from Kenig, Preiss and Toro \cite{KPT}, 
\item the Alt-Caffarelli-Friedman monotonicity
formula \cite{ACF}, and 
\item a rectifiability criterion by Girela-Sarri\'on and Tolsa \cite{GT}, which in turn uses techniques which arise from the solution of the David-Semmes problem by
Nazarov, Tolsa and Volberg \cite{NToV}, \cite{NToV-pubmat} and the work of Eiderman, Nazarov, and Volberg 
\cite{ENV}. 
\end{itemize}

The main new tool that we use in the present paper is a new blowup argument which does not require the CDC
and yields the local convergence in $L^2$ of some subsequences of rescaled Green functions.
This technique was used very recently in \cite{TV} to obtain a new proof of Tsirelson's theorem \cite{Tsirelson}. Instead, the blow argument in \cite{AMT} yields local uniform convergence
of the rescaled Green functions.

\vvv

\section{Harmonic Measure Preliminaries}

We will need the following classical result (see \cite{AHM3TV}, for example):

\begin{lemma}\label{l:w>G}
Let $n\ge 2$ and $\Omega\subset\R^{n+1}$ be a bounded domain. Denote by $\omega^p$ its harmonic measure with pole at $p\in\Omega$ and by $G$ its Green function.
Let $B=B(x_0,r)$ be a closed ball with $x_0\in\partial \Omega$ and $0<r<\diam(\partial \Omega)$. Then, for all $a>0$,
\begin{equation}\label{eq:Green-lowerbound}
 \omega^{x}(aB)\gtrsim \inf_{z\in 2B\cap \Omega} \omega^{z}(aB)\, r^{n-1}\, G(x,y)\quad\mbox{
 for all $x\in \Omega\backslash  2B$ and $y\in B\cap\Omega$,}
 \end{equation}
 with the implicit constant independent of $a$.
\end{lemma}

The next lemma is usually known as Bourgain's estimate. See \cite{AHM3TV} for the precise formulation below.

\begin{lemma}
\label{lembourgain}
There is $\delta_{0}>0$ depending only on $n\geq 1$ so that the following holds for $\delta\in (0,\delta_{0}]$. Let $\Omega\subset \R^{n+1}$ be a  bounded domain, $n-1<s\le n+1$,  $\xi \in \partial \Omega$, $r>0$, and $B=B(\xi,r)$. Then 
\[ \omega^{x}(B)\gtrsim_{n,s} \frac{\mathcal H_\infty^{s}( \delta B \setminus \Omega)}{(\delta r)^{s}}\quad \mbox{  for all }x\in \delta B\cap \Omega .\]
\end{lemma}

\vvv

In the the next lemma we reduce the proof of Theorem \ref{t:main} to the case when the
 domains $\Omega_1,\Omega_2$ are Wiener regular. The proof is from \cite{HMMTV}, but as this paper will not be published, we recreate the details here with some slight modifications.

\begin{lemma}\label{l:weiner}
Let $\Omega_{1}$ and $\Omega_{2}$ be two disjoint connected domains in $\bR^{n+1}$ with harmonic measures $\omega_{i}=\omega_{\Omega_{i}}^{p_{i}}$ for some $p_{i}\in\Omega_{i}$ and suppose $\omega_{1}\ll \omega_{2}\ll \omega_{1}$ on a Borel set $E\subset \d\Omega_{1}\cap \d\Omega_{2}$. Then there are Wiener regular subdomains $\widetilde{\Omega}_{i}\subset \Omega_{i}$ containing $p_{i}$ for $i=1,2$ and $G_{0}\subset E$ with $\widetilde{\omega}_{i}(G_{0})>0$ upon which $\widetilde{\omega}_{2}\ll \widetilde{\omega}_{1}\ll \widetilde{\omega}_{2}$. 
\end{lemma}

\begin{proof}
Let $F_{i}$ be the irregular points for $\Omega_{i}$. By \cite[Theorem 6.6.8]{AG} these sets are polar sets in $\bR^{n+1}$. By \cite[Lemma 6.4.6]{H}, there is a positive superharmonic function $v_{i}$ on $\Omega_{i}$ so that 
\[
\lim_{y\rightarrow x\atop y\in \Omega_{i}} v_{i}(y)=\infty \;\; \mbox{for all }x\in F:=F_{1}\cup F_{2}.\]

%Let 
%\[
%V_{\lambda}=\bigcup_{i=1}^{2} \{x\in \Omega_i:v_{i}(x)>\lambda\}.\]

Let $\lambda>0$. Since $v_{i}$ is superharmonic on $\Omega_{i}$, it is lower semicontinuous, and remains so when we extend it by zero to $\Omega_{i}^{c}$. Thus, for each $x\in F$ there is a closed ball $B_{i}(x)$ centered at $x$ (not containing either $p_{i}$) such that $v_{i}\geq \lambda$ on $B_{i}(x)\cap \Omega_{i}$. Let $\{B_{j}\}$ be a Besicovitch subcovering. Let $H=\bigcup B_{j}$ and
\[
\widetilde{\Omega}_{i} = \Omega_{i}\backslash  H.\]
Note that $\wt\Omega_{i}$ is open. Indeed, to show that $\widetilde{\Omega}_{i}^{c}$ is closed consider $x_{k}\in \wt\Omega_{i}^{c}$, $k\geq 1$ and $x_{k}\rightarrow x$. Then we need to show $x\in \widetilde{\Omega}_{i}^{c}$. If there is a subsequence contained in $\Omega_{i}^{c}$, we are done. Otherwise, assume that $x_{k}\in H\backslash \Omega_{i}^{c}=H\cap \Omega_{i}$. If $x_{k}\in B_{j}$ for infinitely many $k$, then $x\in B_{j}$ and we are done since $B_{j}$ is closed and $B_{j}\subset \widetilde{\Omega}_{i}^{c}$. Otherwise, suppose $x_{k}$ is not in any $B_{j}$ more than finitely many times. By the bounded overlap property, if $j(x_k)$ is such that $x_k\in B_{j(x_k)}$, then $r(B_{j(x_k)})\downarrow 0$ as $k\to\infty$, and since the balls are centered on $F\subset \Omega_{i}^{c}$, $x\in \Omega_{i}^{c}\subset \widetilde{\Omega}_{i}^{c}$, and we are done. Thus, $\widetilde{\Omega}_{i}$ is open. 

Let $\widetilde{\omega}_{i} = \omega_{\widetilde{\Omega}_{i}}^{p_{i}}$. 
 Note that $\widetilde{\Omega}_{i}$ is now a regular domain. Indeed, one need only observe that whenever $\Omega\subset \Omega'$ are two domains and $x\in \d\Omega \cap \d\Omega'$ is regular for $\Omega'$, then it is regular for $\Omega$. Hence, if $x\in \d\widetilde{\Omega}_{i}$, then either $x\in \d B_{i}$ for some $i$, in which case $x$ is regular for $B_{i}^{c}\supset \widetilde{\Omega}_{i}$, or $x\in \d\Omega_{i}\backslash F$, in which case $x$ is regular for $\Omega_{i}\supset \widetilde{\Omega}_{i}$ since it is not in $F$, and either case implies $x$ is regular for $\widetilde{\Omega}_{i}$. 

Let $G=E\backslash H$.  By the maximum principle on $\widetilde{\Omega}_{i}$, since $u/\lambda \geq 1$ on $H$, $\omega_{i}(H)\leq \lambda v_{i}(p_{i})$. Picking
\[
\lambda < \frac{1}{2} \min\{\omega_{i}(E)/v_{i}(p_{i})\},\] 
this gives $\omega_{i}(H)\leq \frac{1}{2} \omega_{i}(E)$ and hence $\omega_{i}(G)>0$. 
Similarly,  by the maximum principle, since $u/\lambda \geq 1$ on $H$, $\widetilde{\omega}_{i}(H)\leq \lambda v_{i}(p_{i})$. Picking 
\[
\lambda < \frac{1}{2} \min\{\omega_{i}(G)/v_{i}(p_{i})\},\] 
this gives 
\[
\widetilde{\omega}_{i}(H)\leq \frac{1}{2} \omega_{i}(G).\]

Moreover, by the maximum principle, and since $\widetilde{\Omega}_{i}$ is a regular domain,
\[
\widetilde{\omega}_{i}(H^{c}\cap G^{c})\leq {\omega}_{i}(H^{c}\cap G^{c}).\]
Thus,
\begin{align*}
\widetilde{\omega}_{i}(G)
& =1- \widetilde{\omega}_{i}(G^{c})
=1-\widetilde{\omega}_{i}(H\cap G^{c})-\widetilde{\omega}_{i}(H^{c}\cap G^{c}) \\
& \geq 1-\frac{1}{2} \omega_{i}(G)-{\omega}_{i}(H^{c}\cap G^{c})
\geq \omega_{i}(G)-\frac{1}{2} \omega_{i}(G)= \frac{1}{2} \omega_{i}(G)>0
\end{align*}

Note that $\widetilde{\omega}_{1}\ll \omega_{1}$ on $G$ by the maximum principle (or by Carleman's principle, see \cite[Theorem 11.3(b)]{HKM}), and since $\widetilde{\omega}_{1}(G)>0$, it is not hard to show using the Lebesgue decomposition theorem that there is $G_{1}\subset G$ of full $\widetilde{\omega}_{1}$-measure upon which we also have $ \omega_{1} \ll \widetilde{\omega}_{1}$. Hence $\omega_{1}(G_{1})>0$, which implies $\omega_{2}(G_{1})>0$. The same reasoning gives us a set $G_{2}\subset G_{1}$ upon which $\widetilde{\omega}_{2} \ll \omega_{2} \ll \widetilde{\omega}_{2}$. Thus, $\widetilde{\omega}_{2}\ll \widetilde{\omega}_{1}\ll \widetilde{\omega}_{2}$ on $G_{2}$. 
\end{proof}
\vv

\begin{remark}
In light of this lemma, for the remainder of this paper, to prove Theorem \ref{t:main} we shall assume our domains 
$\Omega_1,\Omega_2$ are Wiener regular.
\end{remark}
\vv

\section{The Alt-Caffarelli-Friedman monotonicity formula}

The following theorem contains the Alt-Caffarelli-Friedman monotonicity formula:

\begin{theorem} \label{t:ACF} \cite[Theorem 12.3]{CS} Let $B(x,R)\subset \R^{n+1}$, and let $u_1,u_2\in
W^{1,2}(B(x,R))\cap C(B(x,R))$ be nonnegative subharmonic functions. Suppose that $u_1(x)=u_2(x)=0$ and
that $u_1\cdot u_2\equiv 0$. Set
\begin{equation}\label{eq*123}
\gamma(x,r) = \left(\frac{1}{r^{2}} \int_{B(x,r)} \frac{|\grad u_1(y)|^{2}}{|y-x|^{n-1}}dy\right)\cdot \left(\frac{1}{r^{2}} \int_{B(x,r)} \frac{|\grad u_2(y)|^{2}}{|y-x|^{n-1}}dy\right).
\end{equation}
Then $\gamma(x,r)$ is a non-decreasing function of $r\in (0,R)$ and $\gamma(x,r)<\infty$ for all $r\in (0,R)$. That is,
\begin{equation}\label{e:gamma}
\gamma(x,r_{1})\leq \gamma(x,r_{2})<\infty \;\;  \mbox{ for } \;\; 0<r_{1}\leq r_{2}<R.
\end{equation}
\end{theorem}

We remark that the preceding result was also stated in \cite{AMT}, although under somewhat stronger assumptions. In the current paper we will apply the preceding formula to the case when $\Omega_1$ and
$\Omega_2$ are disjoint Wiener regular domains, $x\in\partial \Omega_1\cap \partial \Omega_2$, with $u_1,u_2$ equal to the Green functions of $\Omega_1,\Omega_2$ with poles at $p_1,p_2$, extended by $0$ to $\Omega_1^c,\Omega_2^c$.  In this case, it is well known that $u_i\in W^{1,2}_{loc}(\R^{n+1}\setminus\{p_i\})\cap C(\R^{n+1}\setminus\{p_i\})$ for $i=1,2$ and so the assumptions of 
the preceding theorem are satisfied in any ball which does contain $p_1$ and $p_2$. 

\vv

Arguing as in \cite[Theorem 3.3]{KPT}, we obtain:

\begin{lemma}\label{l:beurling}
Let $\Omega_1,\Omega_2\subset \R^{n+1}$ be as in Theorem \ref{t:main}, and assume further that they are
Wiener regular. For $i=1,2$, let $\omega_i$ be the harmonic measure of $\Omega_i$ with pole at $p_i\in\Omega_i$. Let $0<R<\min_i \dist(p_i,\d\Omega_i)$. Then for for $0<r<R/4$ and $\xi\in \d\Omega_1\cap \d\Omega_2$,
\begin{equation}\label{e:beurling0}
\frac{\omega_i(B(\xi,r))}{r^{n}}\lesssim 
 \left(\frac1{r^{2}} \int_{B(\xi,2r)} 
\frac{|\grad u_i(y)|^{2}}{|y-\xi|^{n-1}}dy\right)^{\frac{1}{2}} 
\lesssim \left(\frac{1}{r^{n+3}} \int_{B(\xi,4r)}|u_i|^{2}\right)^{\frac{1}{2}}
\end{equation}
and in particular,
\begin{equation}\label{e:beurling}
\frac{\omega_1(B(\xi,r))}{r^{n}}\,\frac{\omega_2(B(\xi,r))}{r^{n}}\lec \gamma(\xi,2r)^{\frac{1}{2}},
\end{equation}
where $\gamma(\xi,2r)$ is defined by \rf{eq*123}.
\end{lemma}
\vv

\begin{lemma}\label{l:largeinterior}
Let $\Omega_1\subset \R^{n+1}$ be a Wiener regular domain and denote by $\omega_1$ its harmonic measure with pole at $p_1 \in\Omega_1$. Let $B$ be a ball centered at $\partial\Omega_1$ such that $p_1\not\in 10B$. Suppose
that $\omega_1(4B)\leq C\,\omega_1(\delta_0 B)$ and $\HH^{n+1}(B\setminus\Omega_1)\geq C^{-1}r(B)^{n+1}$.
Then,
$$\HH^{n+1}(\Omega_1\cap2\delta_0 B)\gtrsim r(B)^{n+1}.$$ 
\end{lemma}

\begin{proof}
Let $\vphi$ be a non-negative bump function which equals $1$ on $\delta_0 B$ and is supported  on $2\delta_0 B$.
Then we have:
$$\omega_1(\delta_0 B)\leq \int \vphi \,d\omega_1 = \int \Delta\vphi(y)\,u_1(y)\,dy
\lesssim \frac1{r^2} \,\HH^{n+1}(\Omega_1\cap 2\delta_0 B)\,\sup_{y\in2\delta_0 B} u_1(y),$$
where $u_1$ is the Green function with pole at $p_1$. From Bourgain's estimate (taking into account that
$\HH^{n+1}(B\setminus\Omega_1)\geq C^{-1}r(B)^{n+1}$)
and Lemma \ref{l:w>G}
we deduce that 
$$\sup_{y\in2\delta_0 B} u_1(y)\lesssim \frac{\omega_1(4B)}{r^{n-1}},$$
and so
$$\omega_1(\delta_0 B)\lesssim \frac1{r^{n+1}} \,\HH^{n+1}(\Omega_1\cap 2\delta_0 B)\,\omega_1(4B).$$
Using then that $\omega_1(4B)\leq C\,\omega(\delta_0 B)$, the lemma follows.
\end{proof}

\vv

\begin{lemma}\label{l:dublemma}
Let $\Omega_1,\Omega_2\subset \R^{n+1}$ be as in Theorem \ref{t:main}, and assume further that they are
Wiener regular. For $i=1,2$, let $\omega_i$ be the harmonic measure of $\Omega_i$ with pole at $p_i\in\Omega_i$. 
 Let $B$ be a ball centered at $\partial\Omega_1\cap \partial\Omega_2$ such that $p_1,p_2\not\in 10B$. Suppose
that $\omega_i(4B)\leq C\,\omega_i(\delta_0B)$ for $i=1,2$.
Then,
$$\HH^{n+1}( 2\delta_0B\setminus \Omega_1)\approx \HH^{n+1}(2\delta_0B\setminus \Omega_2)\approx r(B)^{n+1}$$
and
$$\sup_{y\in \delta_0 B} u_i(y)\lesssim \frac{\omega_i(4B)}{r^{n-1}} \quad
\mbox{ for $i=1,2$.}$$
\end{lemma}

\begin{proof}
There exists $1\leq i\leq2$ such that 
\begin{equation}\label{e:2db} \HH^{n+1}(2\delta_{0} B\setminus \Omega_{i})\geq C^{-1}r(B)^{n+1}.
\end{equation}
Suppose this holds for $\Omega_1$, then  we only need to show now that $\HH^{n+1}(2\delta_{0} B\setminus\Omega_2)\geq C^{-1} r(B)^{n+1}$. Since $\delta_{0}<1/2$, $B\supset 2\delta_{0} B$, and hence \eqref{e:2db} implies 
$\HH^{n+1}(B\setminus\Omega_1)\geq C^{-1} r(B)^{n+1}$. From Lemma \ref{l:largeinterior}, we infer that 
$$\HH^{n+1}(2\delta_0 B\setminus\Omega_2)\geq \HH^{n+1}(\Omega_1\cap2\delta_0 B)\gtrsim r(B)^{n+1}.$$
%Also, by assumption,
%$$\HH^{n+1}(\Omega_1^c\cap 2\delta_0 B)\geq \HH^{n+1}(\Omega_1^{c} \cap B)\geq C^{-1} r(B)^{n+1}.$$
%%$$\HH^{n+1}(\Omega_1^c\cap 4\delta_0 B)\geq \HH^{n+1}(\Omega_2\cap4\delta_0 B)\gtrsim r(B)^{n+1}.$$
The second statement in the lemma follows from these estimates in combination with Bourgain's estimate and 
Lemma \ref{l:w>G}.
\end{proof}

\vv
\begin{lemma}
\label{l:otherside}
Let $\Omega_1,\Omega_2\subset \R^{n+1}$ be as in Theorem \ref{t:main}, and assume further that they are
Wiener regular. For $i=1,2$, let $\omega_i$ be the harmonic measure of $\Omega_i$ with pole at $p_i\in\Omega_i$ . Let $0<R<\min_i \dist(p_i,\d\Omega_i)$.
 Let $\xi\in \d\Omega_1\cap \d\Omega_2$ and $r< \delta_0 R/4$, $i=1,2$. 
 Suppose
that $\omega_i(B(\xi,4r))\leq C\,\omega_i(B(\xi,\delta_0 r))$ for $i=1,2$. Then we have
\begin{equation}\label{eq:avgGreen-harm}
\left(\frac{1}{r^{n+1}} \int_{B(\xi,r) \cap \Omega_i}|\nabla u_i|^{2}\right)^{\frac{1}{2}}  \lesssim \left(\frac{1}{r^{n+3}} \int_{B(\xi, 2r)\cap \Omega_i}|u_i|^{2}\right)^{\frac{1}{2}} \lesssim \frac{\omega_i(B(\xi,8 \delta_0^{-1}r))}{r^{n}}.
\end{equation}
In particular, 
\begin{equation}\label{e:otherside}
\gamma(\xi,r)^{\frac{1}{2}}\lec \frac{\omega_1(B(\xi,8\delta_0^{-1}r))}{r^{n}}
\frac{\omega_2(B(\xi,8\delta_0^{-1}r))}{r^{n}},
\end{equation}
where $\gamma(\xi,r)$ is defined by \rf{eq*123}.
\end{lemma}

\begin{proof}
Since $u_i$ vanishes continuously at the boundary of $\d \Omega_i$, we may extend it by zero in $\R^{n+1} \setminus \Omega_i$. Then, as the extended function (which we still denote by $u_i$) is non-negative and subharmonic in $\R^{n+1}$, by Caccioppoli's inequality (which still holds for subharmonic functions) and Lemma \ref{l:dublemma}, we infer that
$$
\left( \int_{B(\xi,r)}|\nabla u_i|^{2}\right)^{\frac{1}{2}}  \lesssim \left(\frac{1}{r^2} \int_{B(\xi, 2r)}(u_i)^{2}\right)^{\frac{1}{2}} \lesssim \omega_i(B(\xi,8 \delta_0^{-1}r)) \,r^{\frac{1-n}{2}}.
$$
 This shows \eqref{eq:avgGreen-harm}, which in turn implies \eqref{e:otherside}.
\end{proof}

\vvv

\section{Tangent measures}\label{s:tan}

For $a\in\R^{n+1}$ and $r>0$, we consider the map
$$T_{a,r}(x) = \frac{x-a}{r}.$$
Note that 
\[
T_{a,r}(B(a,r))=\bB:=B(0,1).\] 
Recall also that, given a Radon measure $\mu$, the notation $T_{a,r}[\mu]$ stands for the image measure of $\mu$ by $T_{a,r}$.
That is,
$$T_{a,r}[\mu](A) = \mu(rA+a),\qquad A\subset\R^{n+1}.$$

Given two Radon measure $\mu$ and $\sigma$, we set
$$F_{B}(\mu,\sigma) = \sup_f \int f\,d(\mu-\sigma),$$
where the supremum is taken over all the $1$-Lipschitz functions supported on $B$. 
For $r>0$, we write
\[
F_{r}(\mu,\nu)= F_{\overline B(0,r)}(\mu,\nu),\qquad
F_{r}(\mu)=F_{r}(\mu,0)=\int (r-|z|)_{+} d\mu.\]

\begin{definition} \cite[Section 2]{Preiss} 
\begin{enumerate}[(a)]
\item A set $\cM$ of non-zero Radon measures in $\bR^{n+1}$ is a {\it cone} if $c\mu\in \cM$ whenever $\mu\in \cM$ and $c>0$.
\item A cone $\cM$ is a {\it d-cone} if $T_{0,r}[\mu]\in \cM$ for all $\mu\in \cM$ and $r>0$.
\item The {\it basis} of a $d$-cone $\cM$ is the set $\{\mu\in \cM: F_{1}(\mu)=1\}$.
\item For a $d$-cone $\cM$, $r>0$, and $\mu$ a Radon measure with $0<F_{r}(\mu)<\infty$, we define the {\it distance} between $\mu$ and $\cM$ as 
%\xavi{I have modified the definition of $d_r$ to guaranty the validity of the identities in \rf{e:ojF}}
\[
d_{r}(\mu,\cM)=\inf\ck{F_{r}\ps{\frac{\mu}{F_{r}(\mu)},\nu}: \nu\in \cM, F_{r}(\nu)=1 }
\]
\end{enumerate}
\end{definition}

\begin{lemma}[\cite{KPT} Section 2]
Let $\mu,\nu$ be Radon measures in $\bR^{n+1}$ and $\cM$ a $d$-cone. For $\xi\in \bR^{n+1}$ and $r>0$,
\begin{enumerate}
\item $T_{\xi,r}[\mu](B(0,s))=\mu(B(\xi,sr))$,
\item $\int fd T_{\xi,r}[\mu] = \int f\circ T_{\xi,r} d\mu$,
\item $F_{B(\xi,r)}(\mu)=r F_{1}(T_{\xi,r}[\mu])$,
\item $F_{B(\xi,r)}(\mu,\nu)=rF_{1}(T_{\xi,r}[\mu],T_{\xi,r}[\nu])$,
\item $\mu_{i}\rightarrow \mu$ weakly if and only if $F_{r}(\mu_{i},\mu)\rightarrow 0$ for all $r>0$,
\item $d_{r}(\mu,\cM)\leq 1$,
\item $d_{r}(\mu,\cM)=d_{1}(T_{0,r}[\mu],\cM)$,
\item if $\mu_{i}\rightarrow \mu$ and $F_{r}(\mu)>0$, then $d_{r}(\mu_{i},\cM)\rightarrow d_{r}(\mu,\cM)$. 
\end{enumerate}
\end{lemma}

\begin{theorem}[\cite{Preiss} Corollary 2.7] Let $\mu$ be a Radon measure on $\bR^{n+1}$, and $\xi\in \supp \mu$. Then $\Tan(\mu,\xi)$ has compact basis if and only if 
\begin{equation}
\label{e:compactdouble}
\limsup_{r\rightarrow 0} \frac{\mu(B(\xi,2r))}{\mu(B(\xi,r))}<\infty.
\end{equation}
In this case, $0\in \supp \nu$ for all $\nu\in \Tan(\mu,\xi)$, and 
\[
\frac{\nu(B(0,2r))}{\nu(B(0,r))}\leq \limsup_{r\rightarrow 0} \frac{\mu(B(\xi,2r))}{\mu(B(\xi,r))}\quad\mbox{ for all }r>0.
\] 
\label{t:compactdouble}
\end{theorem}

\begin{lemma}\cite[Lemma 14.6]{Mattila} Let $\mu$ be a Radon measure on $\bR^{n}$, $\phi$  a non-negative locally integrable function on $\bR^{n+1}$, and $\lambda$ the Radon measure such that
$\lambda(B)= \int_{B}\phi d\mu$ for all Borel sets $B$. Then $\Tan(\mu,x)=\Tan(\lambda,x)$ for $\lambda$-almost all $x\in \bR^{n+1}$.
\label{l:abstan}
\end{lemma}

\begin{lemma} \cite[Theorem 14.3]{Mattila} Let $\mu$ be a Radon measure on $\bR^{n+1}$. If $\xi\in \bR^{n+1}$ and \eqn{compactdouble} holds, then every sequence $r_{i}\downarrow 0$ contains a subsequence such that $T_{\xi,r_{j}\#}\mu/\mu(B(\xi,r_{j}))$ converges to a measure $\nu\in \Tan(\mu,\xi)$. 
\label{l:tanexist}
\end{lemma}

%\begin{theorem}\label{t:main}
%Let $\omega$ be a Radon measure in $\bR^{n+1}$ and suppose $\Tan(\omega,\xi)\subset \cP_{k}$ for some integer $k$. Then $\Tan(\omega,\xi)\subset \cF_{k}$.
%\end{theorem}

%
%\begin{proposition}[\cite{Pr87} Proposition 2.2] Let $\cM$ be a $d$-cone. Then $\cM$ has compact basis if and only if for every $\lambda>1$ there is $\tau>1$ such that 
%\begin{equation}
%F_{\tau r}(\Psi)\leq \lambda F_{r}(\Psi) \mbox{ for every }\Psi\in \cM \mbox{ and }r>0.
%\label{e:compact}
%\end{equation}
% In this case, $0\in \supp \Psi$ for all $\Psi\in \cM$.
%\label{p:compact}
%\end{proposition}

\begin{theorem} \label{t:ttt}
\cite[Theorem 14.16]{Mattila}
Let $\mu$ be a Radon measure on $\bR^{n+1}$. For $\mu$-almost every $x\in \bR^{n+1}$, if $\nu\in \Tan(\mu,x)$, the following hold:
\begin{enumerate}
\item $T_{y,r}[\nu]\in \Tan(\mu,x)$ for all $y\in \supp \nu$ and $r>0$.
\item $\Tan(\nu,y)\subset \Tan(\mu,x)$ for all $y\in \supp \nu$.
\end{enumerate}
\end{theorem}

\vvv

\section{The blowup lemmas}

 For a measure $\mu$, $\xi\in \supp \mu$, $L$ an $n$-plane, and $r>0$, we define
\[
\beta_{\mu,1}^{L}(\xi,r)=\frac{1}{r^{n}}\int_{B(\xi,r)} \frac{\dist(x,L)}{r}d\mu(x)\]
and 
\[
\beta_{\mu,1}(\xi,r)=\inf_{L}\beta_{\mu,1}^{L}(\xi,r)\]
where the infimum is over all $n$-dimensional planes $L$. 

The aim of this section is to prove the following lemma. The proof is a variation on the work in \cite{TV}, which in turn is inspired by previous blowup arguments in \cite{AMT} and \cite{KPT}.

\begin{lemma}\label{l:mainlem}
Let $\Omega_{1},\Omega_{2}\subset \bR^{n+1}$ be disjoint domains and suppose there is $E\subset \d\Omega_{1}\cap \d\Omega_{2}$ upon which we have $\omega_{1}|_{E}\ll \omega_{2}|_{E}\ll \omega_{1}|_{E}$. Fix $\ve<1/100$ and let $E_{m}$ be the set of $\xi\in E$ such that for all $0<r<1/m$ and $i=1,2$ we have 
\begin{equation}\label{e:widub}
\omega_{i}(B(\xi,2r))\leq m\, \omega_{i}(B(\xi,r)),
\end{equation} 
\begin{equation}\label{e:ominondeg}
\cH^{n+1}(B(\xi,r)\cap \Omega_{i})\geq  \frac1m\,r^{n+1},
\end{equation}
and
\begin{equation}\label{e:betasmall}
\beta_{\omega_{1},1}(\xi,r)<\ve \frac{\omega_{1}(\xi,r)}{r^{n}}
\end{equation}
Then 
\begin{equation}\label{eq:EMM}
\omega_1\biggl(E\setminus \bigcup_{m\geq1} E_m\biggr)=0.
\end{equation}
\end{lemma}

Set 
\[E^{*}=\ck{\xi\in E: \lim_{r\rightarrow 0} \frac{\omega_{1}(E\cap B(\xi,r))}{\omega_{1}(B(\xi,r))}= \lim_{r\rightarrow 0} \frac{\omega_{2}(E\cap B(\xi,r))}{\omega_{2}(B(\xi,r))}=1}.\]
By \cite[Corollary 2.14 (1)]{Mattila} and because $\omega_{1}$ and $\omega_{2}$ are mutually absolutely continuous on $E$,
\[\omega_1(E\backslash E^{*})= \omega_{2}(E\backslash E^{*})=0.\]
Also, set 
\begin{multline*}
\Lambda_{1} =\left\{\xi\in E^*\!\!: 0<h(\xi):=\frac{d\omega_{2}}{d\omega_{1}}(\xi)=\lim_{r\rightarrow 0} \frac{\omega_{2}( B(\xi,r))}{\omega_{1}( B(\xi,r))} \right. \\
 =\left. \lim_{r\rightarrow 0} \frac{\omega_{2}(E\cap B(\xi,r))}{\omega_{1}(E\cap B(\xi,r))}<\infty\right\}\end{multline*}
%\[
%\Lambda_{2}=\ck{\xi\in E: \lim_{r\rightarrow 0} \frac{\omega_{2}(E\cap B(\xi,r))}{\omega_{1}(E\cap B(\xi,r))}=\infty}\]
%\[
%\Lambda_{3}=\ck{\xi\in E: \lim_{r\rightarrow 0} \frac{\omega_{2}(E\cap B(\xi,r))}{\omega_{1}(E\cap B(\xi,r))}=0}\]
%\[ \Lambda_{4}=\ck{\xi\in E:\lim_{r\rightarrow 0} \frac{\omega_{2}(E\cap B(\xi,r))}{\omega_{1}(E\cap B(\xi,r))} \mbox{ does not exist}}\]
and 
\[
\Gamma = \ck{\xi\in \Lambda_{1}: \xi \mbox{ is a Lebesgue point for $h$ with respect to }\omega_{1}}.
\]
Again, by Lebesgue differentiation for measures (see \cite[Corollary 2.14 (2) and Remark 2.15 (3)]{Mattila}), $\Gamma$ has full measure in $E^{*}$ and hence in $E$.
\vv

To prove \eqref{eq:EMM}, it suffices to show that for $\omega_{1}$-almost every $\xi\in \Gamma$, we have 
\begin{equation}\label{e:limdub}
\limsup_{r\rightarrow 0} \frac{\omega_{1}(B(\xi,2r))}{\omega_{1}(B(\xi,r))}<\infty,
\end{equation}
\begin{equation}\label{e:limcomp}
\liminf_{r\rightarrow 0} \min_{i=1,2} \frac{\cH^{n+1}(B(\xi,r)\cap \Omega_{i})}{r^{n+1}}>0,
\end{equation}
and
\begin{equation}\label{e:limbeta}
\lim_{r\rightarrow 0} \beta_{\omega_{1},1}(\xi,r) \frac{r^{n}}{\omega_{1}(B(\xi,r))}=0.
\end{equation}
We then use some standard measure theory to find our desired sets $E_m$.

%
%\begin{equation}
%\label{e:zerolambdas}
%\omega_{1}(\Lambda_{2})=\omega_{2}(\Lambda_{3})=\omega_{1}(\Lambda_{4})=\omega_{2}(\Lambda_{4})=0
%\end{equation}

The following is proven in \cite[Lemma 5.8]{AMT}. There we assume a capacity density condition, but the assumption is not used in the proof. 

\begin{lemma}\label{l:samew}
Let $\xi\in \Gamma$, $c_{j}\geq 0$,  and $r_{j}\rightarrow 0$ be such that $\omega^{j}_{1}=c_{j}T_{\xi,r_{j}}[\omega_{1}]\rightarrow \omega^{\infty}_{1}$. Then $\omega^{j}_{2}=c_{j}T_{\xi,r_{j}}[\omega_{2}]\rightarrow h(\xi)\omega^{\infty}_{1}$.
\end{lemma}

\vv
Let 
\[
\cF=\{c\cH^{n}|_{V}: c>0, \;\; V\mbox{ a $d$-dimensional plane containing the origin}\}.\]
It is not hard to show that $\cF$ has compact basis.
\vv

%\begin{lemma}\label{l:capf}
%For this proof we will let $\bB=B(0,1)$. Let $\xi\in \d\Omega_{1}\cap \d\Omega_{2}$ be such that there is $h(\xi)\in (0,\infty)$ so that if $r_{j}\downarrow 0$ and $c_{j}>0$ are such that $c_{j} T_{\xi,r_{j}}[\omega_{1}]\rightarrow \omega_{1}^{\infty}\in \Tan(\omega_{1},\xi)$, then $c_{j} T_{\xi,r_{j}}[\omega_{2}]\rightarrow h(\xi)\omega_{1}^{\infty}$. (This happens for  $\omega_{i}$-a.e. $\xi\in \Gamma$, for example.) Then $\Tan(\omega_{1},\xi)\cap \cF\neq\emptyset$. 
%\end{lemma}

\begin{lemma}\label{l:capf}
For $\omega_1$-a.e.\ $\xi\in\Gamma$, $$\Tan(\omega_{1},\xi)\cap \cF\neq\varnothing.$$ 
\end{lemma}

\begin{proof}Recall that we denote $\bB=B(0,1)$.

Let $\omega^{\infty}_{1}\in \Tan(\omega_{1},\xi)$,
$c_{j}\geq 0$, and $r_{j}\rightarrow 0$ be such that $\omega^{j}_{1}=c_{j}T_{\xi,r_{j}}[\omega_{1}]\rightarrow \omega^{\infty}_{1}$. %{\color{blue} 
As $\omega_{1}^{\infty}\neq 0$, there is $R>0$ so that $\omega_{1}^{\infty}(B(0,R))\neq 0$. Without loss of generality, we will assume $R=1/4$, and we can pick $c_{j}$ so that
\begin{equation}\label{e:w>0}
\omega_{1}^{\infty}\bigl(\tfrac{1}{4}\bB\bigr)=1.
\end{equation}
%}

Let ${\Omega_i^j}=T_{\xi,r_{j}}(\Omega_i)$.  Let $u_{1}(x)=G_{\Omega_{1}}(x,p_1)$ on $\Omega_{1}$ and $u_{1}(x)=0$ on $(\Omega_{1})^{c}$ (since we are assuming Wiener regularity, this is continuous). Set
\[u^{j}_{1}(x)=c_{j}\,u_{1}(xr_{j}+\xi)\,r_{j}^{n-1}.\]
Define $u_{2}$ and $u^{j}_{2}$ similarly.

Without loss of generality, by passing to a subsequence we may assume that 
\begin{equation}\label{e:om+nondeg}
\cH^{n+1}(B(\xi,r_{j})\backslash \Omega_{1})\geq \frac{r_{j}^{n+1}}{2}.
\end{equation}
Thus, for $z\in B(\xi,r_{j})$,
\[
\omega_{\Omega_{1}}^{z}(B(\xi,\delta^{-1}r_{j}))
\gec  \frac{\cH^{n+1}(B(\xi,r_{j})\backslash \Omega_{1}))}{r^{n+1}}
\gec 1.\]
Hence, 
\begin{equation}\label{e:Green-lowerbound2}
 \omega_{1}(B(\xi,\delta^{-1}r_{j}))\gtrsim r_{j}^{n-1}\, u_{1}(x)\quad\mbox{
 for all $x\in B(\xi,r_{j})\cap\Omega_1$,}
 \end{equation}
 and so, 
 \begin{equation}\label{e:Green-lowerbound3}
 \omega^{j}_{1}(B(0,\delta^{-1}))\gtrsim \, u^{j}_{1}(x)\quad\mbox{
 for all $x\in \bB\cap\Omega_1^j$,}
 \end{equation}

By Caccioppoli's inequality for subharmonic functions and the uniform boundedness of $u_1^j$ in $\bB$, we deduce that, for $i=1,2$,
\[
\limsup_{j\rightarrow\infty}
\|\nabla u^j_{1}\|_{L^{2}(\frac12 \bB)}
\lesssim \limsup_{j\rightarrow\infty} \|u^j_{1}\|_{L^2(\bB)}
\lesssim \limsup_{j\rightarrow\infty}  \omega^{j}_1(B(0,\delta^{-1}))
\leq \omega^{\infty}_{1}(\cnj{B(0,\delta^{-1})})
\]

See (3.7) of \cite{KPT} for a similar argument.
By the Rellich-Kondrachov theorem, the unit ball of the Sobolev space $W^{1,2}(\frac12 \bB)$ is relatively
compact in $L^2(\frac12 \bB)$, and thus there exists a subsequence of the functions $u^j_{1}$ which
converges {\em strongly} in $L^2(\frac12 \bB)$ to another function $u_{1}^{\infty}\in L^2(\frac12 \bB)$.
It easy to check that
\begin{equation*}
\int \phi\,  d\omega^{j}_{1} = \int \Delta\phi  \,u^j_{1}\,dx,
\end{equation*}
for any $C^\infty$ function $\vphi$ compactly supported in $\frac{1}{2}\bB$. Then passing to a limit, it follows that
\begin{equation}\label{eq302}
\int \phi\,  d\omega^{\infty}_{1} = \int \Delta\phi  \,u^\infty_{1}\,dx,\,\,\, \text{for any}\,\, \vphi \in C^\infty_c(\tfrac{1}{2}\bB).
\end{equation}

Observe now that %\mih{$||u^{j}_{1}||_{L^{\infty}(\bB )} \longmapsto||u^{j}_{1}||_{L^{\infty}(\bB \cap \Omega_1^{j})}$}
\begin{align*}
1 \stackrel{\eqref{e:w>0}}{=} &\omega_{1}^{\infty}(\frac{1}{4} \bB)
 \leq \int \phi \,d\omega^{\infty}_{1}
=\int_{\Omega_1}u^{\infty}_{1}\Delta \phi \,dx
=\lim_{j} \int_{\Omega_1^{j}} u^{j}_{1} \Delta \phi\, dx\\
& 
\leq \lim_{j} \ps{\int_{\bB\cap \Omega_1^{j} \cap \{u^{j}_{1}>t\}} u^{j}_{1} \Delta \phi\, dx
+ \int_{\bB\cap \Omega_1^{j} \cap \{u^{j}_{1}\leq t\}} u^{j}_{1} \Delta \phi \,dx}\\
& \leq\liminf_{j} \ps {| \{x\in \bB\cap \Omega_1^{j} : u^{j}_{1}>t\} | \cdot ||u^{j}_{1}||_{L^{\infty}(\bB \cap \Omega_1^{j})} ||\Delta \phi||_{L^{\infty}(\bB)}}
+ t ||\Delta \phi||_{L^{\infty}(\bB)} \\
& \stackrel{\eqn{Green-lowerbound3}}{\lec} \liminf_{j} \ps{|\{x\in \bB\cap \Omega_1^{j} : u^{j}_{1}>t\}|\, \omega^{\infty}_{1}\left(\cnj{B(0,\delta^{-1})}\right)
+ t },
\end{align*}
and so, for $t$ small enough,
\[
|\bB\cap \Omega^{j}_{1}|\geq |\{x\in \bB\cap \Omega_1^{j}: u^{j}_{1}(x)>t\}|\gec \omega^{\infty}_{1}(\cnj{B(0,\delta^{-1})})^{-1}.\]
In particular,
\begin{equation}\label{e:om-nondeg}
|B(\xi,r_{j})\backslash \Omega_{2}|\geq |B(\xi,r_{j})\cap \Omega_{1}|\gec r_{j}^{n+1}\omega^{\infty}_{1}(\cnj{B(0,\delta^{-1})})^{-1}.
\end{equation}

Thus, by the same arguments as earlier in proving \eqn{Green-lowerbound3}, we have that for $j$ large,
\begin{equation}\label{e:Green-lowerbound4}
 \omega^{j}_{2}(B(\xi,\delta^{-1}r_j))\gtrsim \, u^{j}_{2}(x)\,\omega^{\infty}_{1}(\cnj{B(\xi,\delta^{-1}r_j)})^{-1}, \quad\mbox{
 for all $x\in B(\xi,r_j)\cap\Omega_{2}$.}
 \end{equation}
 
 Again, we can pass to a subsequence so that $u^{j}_{2}$ converges in $L^{2}(\frac{1}{2}\bB)$ to a function $u^{\infty}_{2}$, and it holds
 \begin{equation}\label{e:intphu-}
\int \phi \, d\omega^{\infty}_{2} = \int \Delta\phi  \,u^\infty_{2}\,dx.
\end{equation}

Now set $u^{\infty}=u^{\infty}_{1}-h(\xi)^{-1}u^{\infty}_{2}$. Then by Lemma \ref{l:samew},
\begin{align*}
\int u^{\infty} \Delta \phi &=\int  u^{\infty}_{1} \Delta \phi  - h(\xi)^{-1}\int  u^{\infty}_{2} \Delta \phi =\int \phi \, d\omega^{\infty}_{1}  - h(\xi)^{-1} \int \phi \, d\omega^{\infty}_{2} \\
&=\int \phi \, d\omega^{\infty}_{1}-  h(\xi)^{-1} h(\xi) \int \phi \, d\omega^{\infty}_{1}  =0,
\end{align*}
for all $\phi\in C_{c}^{\infty}(\frac{1}{2}\bB)$. Therefore, $u^{\infty}$ is harmonic in $\frac{1}{2}\bB$. 

Next we claim that $u^\infty\not\equiv0$ and that
\begin{equation}
\tfrac{1}{2}\bB\cap \supp \omega^{\infty}_1=\{u^{\infty}=0\}\cap \tfrac{1}{2} \bB.
\end{equation}
First note that as $u_{i}^{j}\rightarrow u_{i}^{\infty}$ in $L^{2}(\frac{1}{2}\bB)$ and $u^{j}_{i}$ have disjoint supports for all $j$, we know that 
\[
0=\lim_{j\rightarrow\infty} \int_{\frac{1}{2}\bB} u_{1}^{j}u_{2}^{j}dx
=\int_{\frac{1}{2}\bB} u_{1}^{\infty}u_{2}^{\infty}dx\]
and so $u_{1}^{\infty}$ and $u_{2}^{\infty}$ cannot be nonzero simultaneously in $\frac{1}{2}\bB$, 
except in a set of zero Lebesgue measure.
Since $u_1^\infty\not\equiv0$ (by \rf{eq302}), this implies that $u^\infty\not\equiv0$. %\mih{I am not totally convinced about this part of the argument. Somehow we say that either the one is zero or the other apart from a set with Lebesgue measure zero. But we know that the set where the harmonic function $u^\infty$ vanishes in a ball has Lebesgue measure zero since its Hausdorff measure is $\sim r^n$.}
Another consequence is that, in $\frac{1}{2}\bB$,
\begin{equation}\label{e:uposneg}
u_1^\infty = u^\infty\,\chi_{\{u^\infty>0\}} \quad\mbox{and}\quad u_2^\infty = - h(\xi)\,u^\infty\,\chi_{\{u^\infty<0\}},
\end{equation}
which in particular implies that $u_1^\infty$ and $u_2^\infty$ are continuous in $\frac{1}{2}\bB$, because $u^\infty$ is harmonic there.

Observe now that for each $i=1,2$,
\begin{equation}\label{e:du>0}
\supp \omega_{1}^{\infty} \cap \frac{1}{2}\bB= \d\{u_{i}^{\infty}>0\}\cap \frac{1}{2}\bB.
\end{equation}
This is essentially proven in \cite[Lemma 4.7]{AAM16}. We omit the details. Thus, we have 
\[
\supp \omega_{1}^{\infty} \cap \frac{1}{2}\bB\subset  \{u_{1}^{\infty}=0\} \cap \{u_{2}^{\infty}=0\}\cap \frac{1}{2}\bB
\subset
\{u^{\infty}=0\}\cap \frac{1}{2}\bB.\]

For the converse inclusion, note that
 if $x\in \frac{1}{2} \bB$ and $u^{\infty}(x)=0$, then $u_{1}^{\infty}(x)=u_{2}^{\infty}(x)=0$, since $u^{\infty}=u_{1}^{\infty}-h(\xi)^{-1}u_{2}^{\infty}$ and we have just shown that
$u_{1}^{\infty}$ and $u_{2}^{\infty}$ cannot be positive simultaneously. Further,  $u^{\infty}$ cannot vanish identically in any ball containing $x$ in $\frac{1}{2}\bB$ (because it is harmonic and not identically $0$), and thus either $u_{1}^{\infty}$ or $u_{2}^{\infty}$ must be positive in that ball. These two facts imply
\[
\{u^{\infty}=0\}\cap \frac{1}{2}\bB\subset \ps{\d\{u_{1}^{\infty}>0\} \cup \d\{u_{2}^{\infty}>0\}}\cap \frac{1}{2}\bB =\supp \omega_{1}^{\infty}\cap \frac{1}{2}\bB.\]
This proves the claim.

In particular, $\frac{1}{2}\bB\cap \supp \omega^{\infty}_1$ is a smooth real analytic variety. Then, 
arguing as in \cite{AMT}, for example, one deduces that
$$d\omega_1^\infty|_{\frac{1}{2}\bB}= -c_n(\nu_{\Omega^{\infty}_1} \!\cdot\! \nabla u^\infty_1)\, d\cH^n|_{
\d^*\Omega^{\infty}_1\cap \frac{1}{2}\bB},$$
where $\d^*\Omega^{\infty}_1$ is the reduced boundary of $\Omega^{\infty}_1=
\{u_1^\infty>0\}$ and $\nu_{\Omega^{\infty}_1}$ is the measure theoretic outer unit normal.
Hence,  
$\omega^{\infty}_{1}$ is absolutely continuous with respect to surface measure of $\d\Omega^{\infty}_1$ in $\frac{1}{2}\bB$. Thus, since the tangent measure at $\cH^n$-almost every point of $\d\Omega_{1}^{\infty}$ is contained in $\cF$, using \Lemma{abstan}, we can take another tangent measure of $\omega^{\infty}_{1}$ that is in $\cF$ and apply Theorem \ref{t:ttt}
\end{proof}
\vv

%The previous lemma implies that
%\begin{equation}\label{e:tancapf}
%\Tan(\omega_{1},\xi)\cap \cF\neq\emptyset \;\;  \mbox{ whenever } \;\; \Tan(\omega_{1},\xi)\neq\emptyset  \mbox{ and }\xi\in \Gamma.
%\end{equation}

The following lemma has an identical proof to that of \cite[Lemma 5.11]{AMT}.

\begin{lemma}\label{l:tanconnect}
Let $\Omega_{1}$ and $\Omega_{2}$ be as above and let $\xi\in\Gamma$. If $\Tan(\omega_{1},\xi)\cap \cF\neq\varnothing$, then $$\lim_{r\rightarrow 0} d_{1}(T_{\xi,r}[\omega_{1}],\cF)= 0.$$ In particular, $\Tan(\omega_{1},\xi)\subset \cF$. 
\end{lemma}

By \Theorem{ttt} and \Lemma{tanconnect}, $\Tan(\omega_{1},\xi)\subset \cF$ for $\omega_{1}$ a.e. $\xi\in \Gamma$. By Theorem \ref{t:compactdouble}, $\omega_{1}$ and $\omega_{2}$ are pointwise doubling at each such point, which proves \eqn{limdub}. Also, \Lemma{dublemma} implies \eqn{limcomp}. We will now show that \eqn{limbeta} holds for such a $\xi$.

%\vv
%\begin{lemma}
%For $\xi\in F$, 
%\begin{equation}\label{e:beta1tozero}
%\lim_{r\rightarrow 0} \beta_{\omega_{1},1}(\xi,r) \frac{r^{n}}{\omega_{1}(B(\xi,r))}=0.
%\end{equation}
%\end{lemma}
%
%\begin{proof}
 Let $\omega_{r}= T_{\xi,r}[\omega_{1}]$.  By the compactness of $\cF$ and the definition of $d_{1}$, there is an $n$-plane $V$ such that, if $\mu=\cH^{n}|_{V}/F_{1}(\cH^{n}|_{V})$, then 
\begin{equation}\label{e:F1/F1}
\lim_{r\rightarrow 0} F_{1}(\omega_{r}/F_{1}(\omega_{r}),\mu)=0.
\end{equation}
Let $\phi$ be a $2$-Lipschitz function which equals $1$ on $\frac{1}{2}\bB$ and $0$ on $\bB^{c}$, and set $\psi=\dist(x,V)\phi$. Note that for $r<r_{0}/2$, \eqn{widub} implies $F_{1}(\omega_{r})\lec \omega_{r}(\frac{1}{2}\bB)$, and so
\begin{align*}
 F_{1}(\omega_{r}/F_{1}(\omega_{r}),\mu)
 & \gec F_{1}(\omega_{r})^{-1}\int_{\bB} \psi (x)\,d\omega_{r}^{1}(x)- \int_{\bB} \psi (x)\,d\mu(x) \\
 & \gec \,\avint_{\frac{1}{2}\bB}\dist(x,V)\,d\omega_{r}(x)-0 \\
 & =\,\avint_{B(\xi,r/2)} \frac{\dist(x,rV+\xi)}{r}\,d\omega_{1}(x)
 \geq  \frac{(r/2)^{n}}{\omega_{1}(B(\xi,r/2))}\,\beta_{\omega_{1},1}(\xi,r/2).
\end{align*}
This and \eqn{F1/F1} imply \eqn{limbeta}.

%Take any $r_{j}\rightarrow 0$. We can pass to a subsequece so that $T_{\xi,r_{j}}[\omega_{1}]$ converges weakly to $\omega_{1}^{\infty}$. By \Lemma{samew}, $T_{\xi,r_{j}}[\omega_{2}]\rightarrow h(\xi)\omega_{1}^{\infty}$. By \eqref{e:beurling} and \eqn{widub}, this implies there is $r_{\xi}>0$ so that for all $r<r_{\xi}$,
%\[
%\frac{\omega_{1}(B(\xi,r))}{r^{n}}\lec h(\xi)^{-1} \gamma(\xi,r)\leq h(\xi)^{-1} \gamma(\xi,r_{\xi})<\infty\]
%since $\gamma(\xi,r)$ is increasing by Theorem \ref{t:ACF}. Thus, \eqref{e:F1/F1} implies
%\begin{align*}
%\lim_{j} \beta_{\omega_{1},1}(\xi,r_{j}/2)
%& \leq  \lim_{j} \frac{\omega_{1}(B(\xi,r_{j}/2))}{(r_{j}/2)^{n}} F_{1}(\omega_{r_{j}}/F_{1}(\omega_{r_{j}}),\mu)\\
%& \lec \lim_{j}h(\xi)^{-1}\gamma(\xi,r_{\xi}) F_{1}(\omega_{r_{j}}/F_{1}(\omega_{r_{j}}),\mu)=0.
%\end{align*}
%We have thus shown that every sequence $r_{j}\downarrow 0$ has a subsequence such that $\lim_{j} \beta_{\omega_{1},1}(\xi,r_{j}/2)=0$, which implies $\lim_{r\downarrow 0} \beta_{\omega_{1},1}(\xi,r/2)=0$, which proves the lemma.
%\end{proof}

To conclude the proof of \Lemma{mainlem}, for $j,k\in \bN$, set
\begin{multline*}
E_{j,k}=\Bigl\{\xi\in \Gamma: \omega_{i}(B(\xi,2r))\leq k\, \omega_{i}(B(\xi,r)), \\
\cH^{n+1}(B(\xi,r)\cap \Omega_{i})\geq k^{-1} r^{n+1}, \mbox{ and }\\
\beta_{\omega_{1},1}(\xi,r) \frac{r^{n}}{\omega_{1}(B(\xi,r))} <\ve  \mbox{ for }i\in \{1,2\}, \;\; 0<r<1/j\Bigr\}.
\end{multline*}
Then we have shown above that almost every $\xi\in \Gamma$ lies in one of these sets, and so there must be one for which $\omega_{1}(E_{j,k})>0$. Setting $F=E_{j,k}$, $r_{0}=1/j$, and $C=c^{-1}=j$ finishes the proof of \Lemma{mainlem}.

\vvv

\section{Riesz transforms}

In this section we will complete the  proof of Theorem \ref{t:main} under the additional
assumption that both $\Omega_1$ and $\Omega_2$ are Wiener regular. So given 
 $E\subset \d\Omega_1\cap \d\Omega_2$ so that $\omega_1\ll\omega_2\ll\omega_1$ on $E$, we have to show that $E$ contains an $n$-rectifiable subset $F$ on which $\omega_1,\omega_2$ are mutually absolutely continuous with respect to $\cH^{n}$. 

 Reducing $E$ if necessary, we may assume that
$\diam(E)\leq \frac1{10}\,\min(\diam(\Omega_1),\diam(\Omega_2)$.
Let $\wt B$ be some ball centered at $E$ with radius $r(\wt B)=2\diam(E)$. We choose the poles $p_i$ for $\omega_i$
so that $p_i\in\Omega_i\cap 2\wt B\setminus \wt B$.
Further, by interchanging $\Omega_1$ and $\Omega_2$ if necessary, we may assume also that
$$\HH^{n+1}(\wt B\setminus \Omega_1)\approx \HH^{n+1}(\wt B),$$
so that, by Lemma \ref{lembourgain}
\begin{equation}\label{eqbour10}
\omega_1(2\delta^{-1}\wt B) = \omega_1^{p_1}(2\delta^{-1}\wt B)\approx 1.
\end{equation}

 \vv
 
 Given $\gamma>0$, a Borel measure $\mu$ and a ball $B\subset\R^{n+1}$, we denote
$$P_{\gamma,\mu}(B) = \sum_{j\geq0} 2^{-j\gamma}\,\Theta_\mu(2^jB),$$
where $\Theta_\mu(B) = \frac{\mu(B)}{r(B)^n}$.

\vv

Given $a,\gamma>0$, we say that a ball $B$ is {\it $a$-$P_{\gamma,\mu}$-doubling} if
$$P_{\gamma,\mu}(B) \leq a \,\Theta_\mu(B).$$

\begin{lemma}\label{l:adoubling}
There is $\gamma_{0}\in (0,1)$ so that the following holds. Let $\Omega\subset \bR^{n+1}$ be any domain and $\omega$ its harmonic measure. For all $\gamma>\gamma_{0}$, there exists some big enough constant $a=a(\gamma,n)>0$ such that  for $\omega$-a.e.
$x\in\R^{n+1}$ there exists a sequence of $a$-$P_{\gamma,\omega}$-doubling balls $B(x,r_i)$, with
$r_i\to0$ as $i\to\infty$.
%In particular,  there exists some big enough constant $a>0$ depending only on $n$ such that for $\omega$-a.e. $x\in\R^{n+1}$ there exists a sequence of $a$-$P_{1,\omega}$-doubling balls $B(x,r_i)$, with
%$r_i\to0$ as $i\to\infty$.
\end{lemma}

\vv

From now on we assume that $a$ and $\gamma$ are fixed constants such that for any domain $\Omega\subset\R^{n+1}$, for
$\omega$-a.e.
$x\in\R^{n+1}$ there exists a sequence of $a$-$P_{\gamma,\omega}$-doubling balls $B(x,r_i)$, with
$r_i\to0$ as $i\to\infty$.

\vv
Recall that the harmonic measures $\omega_1$ and $\omega_2$ are mutually absolutely continuous on $E\subset\d\Omega_1\cap \d\Omega_2$, and that $h$ denotes the density function $h(\xi)=\frac{d\omega_{2}}{d\omega_1}(\xi)$ and that we assume that $\Omega_1$, $\Omega_2$ are Wiener regular.\vv

Let $E_{m}$ be one of the sets from \Lemma{mainlem} and fix $m\geq 1$ so that $\omega_{1}(E_{m})>0$.

\vv

\begin{lemma}\label{lem:ver1}
Let $m\geq1$ and $\delta>0$. For $\omega_1$-a.e.\ $x\in E_m$, there is $r_{x}>0$ so that 
for any $a$-$P_{\gamma,\omega_1}$-doubling ball $B(x,r)$ with radius $r\leq r_x$ there exists a subset $G_m(x,r)\subset E_m\cap B(x,r)$ such that 
\begin{equation}\label{e:mainl1}
\frac{\omega_1(B(z,t))}{t^n} \lec \frac{\omega_1(B(x,r))}{r^n} \quad\mbox{for all $z\in G_m(x,r)$, $0<t\leq 2r$,}
\end{equation}
and so that $\omega_1 (B(x,r) \setminus G_m(x,r)) \leq \delta \,\omega_1(B(x,r))$.
\end{lemma}

The proof is almost the same as the one of the analogous Lemma 6.2 from \cite{AMT}.
The only change  is that we cannot rely on
 Lemma 4.11 from \cite{AMT}, and instead we use the fact that,
  by Lemmas \ref{l:beurling}
and \ref{l:otherside}, given
$\xi\in\partial \Omega_1\cap\partial\Omega_2$ and $0<s<r$, with $r$ small enough, if $\omega_i(B(\xi,4r))\leq C\,\omega_i(B(\xi,\delta_0 r))$ for $i=1,2$ (which is guarantied by Lemma \ref{l:mainlem}),  then   we have
$$
\gamma(\xi,s)^{\frac{1}{2}}\leq \gamma(\xi,r)^{\frac{1}{2}}\lec \frac{\omega_1(B(\xi,8\delta_0^{-1}r))}{r^{n}}
\frac{\omega_2(B(\xi,8\delta_0^{-1}r))}{r^{n}}.$$

\vv

Given $m\geq1$ and $\delta>0$, we denote by $\wt E_{m,\delta}$ the subset of points $x\in E_m$ for which there exists $r_x>0$ as in Lemma \ref{lem:ver1}, so that $\omega_1\bigl(E_m\setminus \wt E_{m,\delta}\bigr)=0$. 

\vv

\begin{lemma}\label{lem:ver2}
Let $m\geq1$ and $\delta>0$. Let $x_0\in \wt E_{m,\delta}$ and $$0<r_0\leq \min(r_{x_0},1/m,c_1\dist(p_1,
\d\Omega_1)),$$ for some $c_1>0$ small enough (recall that $\omega_i$ is the harmonic measure for $\Omega_i$ with pole at $p_i$).
Suppose that the ball $B_{0}=B(x_{0},r_{0})$ is $a$-$P_{\gamma,\omega_1}$-doubling. Then for 
all $x \in G_m(x_{0},r_{0})$ it holds that
\begin{equation}\label{e:mainl2}
\RR_{*}(\chi_{2B_{0}}\omega_1)(x)\lesssim\Theta_{\omega_1}(B_{0}).
\end{equation}
\end{lemma}

In the proof of the analogous lemma in \cite{AMT} we used the fact that, for small radii,
the $\beta_\infty$ coefficients of the boundary for balls centered at $E$ are small $\omega_1$-a.e.\ in the case
that $\Omega_1$ and $\Omega_2$ satisfy the CDC. This is no longer true (as far as we know), and so the arguments below are somewhat different (in fact, they are inspired by the estimates
in the Key Lemma 4.3 from \cite{AHM3TV}).

\begin{proof}
To estimate  $|\RR_r(\chi_{2B_0}\omega_1)(x)|$ for $x \in G_m(x_{0},r_{0})$ we may assume that $r\leq r_0/4$ because 
 $|\RR_r(\chi_{2B_0}\omega_1)(x)|=0$ if $r\geq 4r_0$ and \rf{e:mainl2} is trivial
 in the case $r_0/4<r<4r_0$.

So we take  $x \in G_m(x_{0},r_{0})$ and $0<r\leq r_0/4$. 
Note that 
\begin{align*}
|\RR_r(\chi_{2B_0}\omega_1)(x)| &= |\RR_r \omega_1(x) - \RR_r(\chi_{(2B_0)^c}\omega_1)(x)|\\
& \leq |\RR_r \omega_1(x) - \RR_{r_0/4}\omega_1(x)| + 
|\RR_{r_0/4}\omega_1(x) - \RR_r(\chi_{(2B_0)^c}\omega_1)(x)|.
\end{align*}
It is immediate to check that the last term is bounded above by $C \Theta_{\omega_1}(2B_0)$,
and thus
\begin{equation}\label{eqas0}
|\RR_r(\chi_{2B_0}\omega_1)(x)| \leq |\RR_r \omega_1(x) - \RR_{r_0/4}\omega_1(x)| +
C \Theta_{\omega_1}(B_0).
\end{equation}

Let $\vphi:\R^{n+1}\to[0,1]$ be a radial $\CC^\infty$ function  which vanishes on $B(0,1)$ and equals $1$ on $\R^{n+1}\setminus B(0,2)$,
and for $\ve>0$ and $z\in \R^{n+1}$ denote
$\vphi_\ve(z) = \vphi\left(\frac{z}\ve\right) $ and $\psi_\ve = 1-\vphi_\ve$.
We set
$$\wt\RR_\ve\omega_1(z) =\int K(z-y)\,\vphi_\ve(z-y)\,d\omega_1(y),$$
where $K(\cdot)$ is the kernel of the $n$-dimensional Riesz transform. 
Note that
$$\bigl|\wt\RR_{r}\omega_1(x)- \RR_{r}\omega_1(x)\bigr|\lesssim \Theta_{\omega_1}(B(x,2r)).$$
Therefore, by \rf{eqas0} and \rf{e:mainl1},
\begin{equation}\label{eqas1}
|\RR_r(\chi_{2B_0}\omega_1)(x)| \leq |\wt\RR_r \omega_1(x) - \wt\RR_{r_0/4}\omega_1(x)| +
C \Theta_{\omega_1}(B_0).
\end{equation}

To estimate the first term on the right hand side of the inequality above, for a fixed $x \in G_m(x_{0},r_{0})$ and $z\in \R^{n+1}\setminus \bigl[\supp(\vphi_r(x-\cdot)\,\omega_1)\cup \{p_1\}\bigr]$, consider the function
$$v_r(z) = \EE(z-p_1) - \int \EE(z-y)\,\vphi_r(x-y)\,d\omega_1(y),$$
so that, by Remark 3.2 from \cite{AHM3TV},
\begin{equation}\label{eqfj33}
u_1(z) = G_{\Omega_1}(z,p_1) = v_r(z) - \int \EE(z-y)\,\psi_r(x-y)\,d\omega_1(y)\quad \mbox{ for $m$-a.e.   $z\in\R^{n+1}$.}
\end{equation}

Since the kernel of the Riesz transform is
\begin{equation}\label{eqker}
K(x) = c_n\,\nabla \EE(x),
\end{equation}
for a suitable absolute constant $c_n$, we have
$$\nabla v_r(z) = c_n\,K(z-p) - c_n\,\RR(\vphi_{ r}(\cdot-x)\,\omega_1)(z).$$

In the particular case $z=x$ we get
$$\nabla v_r(x) = c_n\,K(x-p) - c_n\,\wt\RR_r\omega_1(x).$$
Using this identity also for $r_0/4$ instead of $r$, we obtain
\begin{equation}\label{eqcv1}
|\wt\RR_r \omega_1(x) - \wt\RR_{r_0/4}\omega_1(x)| \approx |\nabla v_r(x) - \nabla v_{r_0/4}(x)|
\leq |\nabla v_r(x)| + |\nabla v_{r_0/4}(x)|
\end{equation}
Since $v_r$ is harmonic in $\R^{n+1}\setminus \bigl[\supp(\vphi_r(x-\cdot)\,\omega_1)\cup \{p_1\}\bigr]$ (and so in $B(x,r)$),  
we have
\begin{equation}\label{eqcv2}
|\nabla v_r(x)| \lesssim \frac1r\,\avint_{B(x,r)}|v_r(z)|\,dm(z).
\end{equation}
From the identity \rf{eqfj33} we deduce that
\begin{align}\label{eqcv3}
|\nabla v_r(x)| &\lesssim \frac1r\,\avint_{B(x,r)}u_1(z)\,dm(z) + 
\frac1r\,\avint_{B(x,r)}
\int \EE(z-y)\,\psi_r(x-y)\,d\omega_1(y)\,dm(z) \\
& =:I + II.\nonumber
\end{align}
To estimate the term $II$ we use Fubini and the fact that $\supp\vphi_r\subset B(x,2r)$:
\begin{align*}
II & \lesssim \frac1{r^{n+2}}\, \int_{y\in B(x,2r)}\int_{z\in B(x,r)} \frac 1{|z-y|^{n-1}} \,dm(z)\,d\omega_1(y) \lesssim \frac{\omega_1(B(x,2r))}{r^{n}}\lesssim \Theta_{\omega_1}(B_0).
\end{align*}

We intend to show now that $I\lesssim \Theta_{\omega_1}(B_0)$.
Clearly it is enough to show that
\begin{equation}\label{eqsuf1}
\frac1r\,| u_1(y)|\lesssim \Theta_{\omega_1}(B_0)\qquad\mbox{for all $y\in  B(x,r)\cap\Omega$.}
\end{equation}
To prove this, observe that by Lemma \ref{l:w>G} (with $B= B(x,r)$, $a=2\delta_0^{-1}$), for all $y\in B(x,r)\cap\Omega$,
we have
$$\omega_1(B(x,2\delta_0^{-1}r))\gtrsim \inf_{z\in B(x,2r)\cap \Omega} \omega^{z}(B(x,2\delta_0^{-1}r))\, r^{n-1}\,|u_1(y)|.$$
On the other hand, by Lemma \ref{lembourgain}, for any $z\in B(x,2r)\cap\Omega_1$,
$$\omega_1^{z}(B(x,2\delta_0^{-1}r))\gtrsim \frac{\HH^{n+1}(B(x,2r)\setminus \Omega_1)}{r^{n+1}}\gtrsim 1.$$
Therefore, $
\omega_1(B(x,2\delta_0^{-1}r))\gtrsim 
 r^{n-1}\,|u_1(y)|.
$
Then,
$$\frac1r\,| u_1(y)|\lesssim
\Theta_{\omega_1}(B(x,2\delta_0^{-1}r))\lesssim P_{\gamma,\omega_1}(B_0)\lesssim \Theta_{\omega_1}(B_0),$$ 
which proves \rf{eqsuf1}.

By the estimates obtained above for the terms $I$, $II$ for $r$ and $r_0/4$, we derive
$$|\nabla v_r(x)| + |\nabla v_{r_0/4}(x)| \lesssim \Theta_{\omega_1}(B_0).$$
Hence, by \rf{eqas1} and \rf{eqcv1}, we infer that
$$|\RR_{r}(\chi_{2B_0}\omega_1)(x)| \lesssim \Theta_{\omega_1}(B_0),$$
as wished.
\end{proof}
\vv

Let $m\geq1$, $\delta>0$, and $x_0\in \wt E_{m,\delta}$, and denote
$$G_m^{zd}(x_0,r_0) = \{x\in G_m(x_0,r_0): \lim_{r\to0}\Theta_{\omega_1}(B(x,r))=0\},$$
and
$$G_m^{pd}(x_0,r_0) = \{x\in G_m(x_0,r_0): \limsup_{r\to0}\Theta_{\omega_1}(B(x,r))>0\}.$$
The notation ``$zd$'' stands for ``zero density'', and ``$pd$'' stands for ``positive density''.

\vv

The proof of the next lemma is the same as the one of the analogous lemma in \cite{AMT}. This is an easy consequence of the
main result from \cite{AHM3TV}, which in turn relies on \cite{NToV} and \cite{NToV-pubmat}.

\begin{lemma}\label{lem:rectifac}
Let $m\geq1$ and $\delta>0$. Let $x_0\in \wt E_{m,\delta}$ and $$0<r_0\leq \min(r_{x_0},1/m,c_1\dist(p_1,
\d\Omega_1)),$$ for some $c_1>0$ small enough.
Suppose that the ball $B_{0}=B(x_{0},r_{0})$ is $a$-$P_{\gamma,\omega_1}$-doubling. Then there is 
an $n$-rectifiable
set $F(x_0,r_0)\subset G_m^{pd}(x_0,r_0)$ such that 
\[\omega_1(G_m^{pd}(x_0,r_0)\setminus F(x_0,r_0))=0\] 
and so that
$\omega_1|_{F(x_0,r_0)}$ and $\HH^n|_{F(x_0,r_0)}$ are mutually absolutely continuous.
\end{lemma}

\vv

\begin{lemma}\label{lem:mainl3}
Let $m\geq1$ and $\delta>0$. Let $x_0\in \wt E_{m,\delta}$ and $$0<r_0\leq \min(r_{x_0},1/m,c_1\dist(p_1,
\d\Omega_1)),$$ for some $c_1>0$ small enough.
Suppose that the ball $B_{0}=B(x_{0},r_{0})$ is $a$-$P_{\gamma,\omega_1}$-doubling, then 
\begin{multline}\label{e:mainl3}
\int_{G_m^{zd}(x_0,r_0)}|\RR\omega_1(x)-m_{\omega_1,G_m^{zd}(x_0,r_0)}(\RR\omega_1)|^2\,d\omega_1(x)\\
\lesssim \left(\frac{r_0}{|x_0 - p_1|}\right)^{2-2\gamma}\,  \Theta_{\omega_1}(B_0)^2\, \omega_1(B_0).
\end{multline}
\end{lemma}

\begin{proof}
We claim that for $\omega_1$-a.e.\, $x\in G_m^{zd}(x_0,r_0)$, 
\begin{equation}
\RR\omega_1(x) = K(x-p_1).
\label{e:r=k2}
\end{equation}
Indeed, consider a sequence $r_{j}\rightarrow 0$ so that $B_{j}=B(x,r_{j})$ is $a$-$P_{\omega_1}$-doubling for every $j$. Then if $\wt\RR_{r}$ is as in the proof of Lemma \ref{l:adoubling}, we have that 
%\xavi{I changed the notation $u_{r_j}$ by $v_{r_j}$, to avoid confusions with the Green functions.}
$\wt\RR_{r_j} \omega_1(x) - K(x - p_1) = c_n \nabla v_{r_j}(x)$. From the exact same estimates we can prove that 
$$|\nabla v_{r_j}(x)| \lesssim \frac{\omega_1(B(x,2 \delta_0 ^{-1} r_j))}{r_j^n} \lesssim \frac{\omega_1(B(x, r_j))}{r_j^n},$$
 where in the last inequality we used 
% \eqref{e:widub}%
%\xavi{I said this is because the ball is P-doubling instead of (5.1).}
that $B_{j}=B(x,r_{j})$ is $a$-$P_{\omega_1}$-doubling. Therefore, taking $j \to \infty$ and since $\lim_{r\rightarrow 0}\Theta_{\omega_1}(B(x,r))=0$, we infer that $\RR\omega_1(x) =K(x-p_1)$.
 Since the distance between the center $\wt z$ of $\wt B$ (this is the ball introduced at the beginning of this section) and $x_0$
is at most $\diam(E)$ and $|\wt z-p_1|\geq 2\diam(E)$, it follows easily that $\delta_0^{-1}\wt B\subset B(x_0, 10\delta_0^{-1}|x_0 - p_1|)$, and then from \rf{eqbour10} we infer that
$$\omega_1(B(x_0, 10\delta_0^{-1}|x_0 - p_1|))\approx1.
$$
Then we obtain 
 \begin{align*}
|\RR\omega_1(x) &- m_{\omega_1,G_m^{zd}(x_0,r_0)}(\RR\omega_1)| \leq \!\!\sup_{y\in G_m^{zd}(x_0,r_0)} \!|K(x-p_1) - K(y-p_1)|\\
&\lesssim \frac{r_0}{|x_0-p_1|^{n+1}} \lesssim \frac{r_0}{|x_0-p_1|} \frac{\omega_1(B(x_0, 10\delta_0^{-1}|x_0 - p_1|))}{|x_0 - p_1|^{n}}  \\
& =\Theta_{\omega_1}(B(x_0,10\delta_0^{-1}|x_0-p_1|))\, \left(\frac{r_0}{|x_0-p_1|}\right)^\gamma \left(\frac{r_0}{|x_0-p_1|}\right)^{1-\gamma}\\
& \lesssim P_{\gamma,\omega_1}(B_0)\left(\frac{r_0}{|x_0-p_1|}\right)^{1-\gamma}\\
& \lesssim \Theta_{\omega_1}(B_0)\left(\frac{r_0}{|x_0-p_1|}\right)^{1-\gamma}.
\end{align*}
where in the last inequality we used the fact that $B_0$ is $a$-$P_{\gamma,\omega_1}$-doubling. The conclusion of the lemma readily follows.
 \end{proof}
\vv

To finish the proof of Theorem \ref{t:main} we proceed as in \cite{AMT}:  by combining the preceding lemma with the main result
from \cite{GT}, it follows easily that $\omega_1(G_m^{zd}(x_0,r_0))=0$. From this fact and Lemma \ref{lem:rectifac}, one
deduces Theorem \ref{t:main}.
The precise arguments are the same as the ones in the end of Section 6 of
\cite{AMT}.

\vvv
We have now proven Theorem \ref{t:main} for regular domains. We now apply \Lemma{weiner} to obtain the general case. This finishes the proof of Theorem \ref{t:main}.

\end{document}